\documentclass{article}
\usepackage[mathscr]{eucal}
\usepackage{amsmath,amssymb,amsbsy,amsthm,amsfonts,authblk,comment}
\usepackage[textwidth=125truemm,textheight=185truemm]{geometry}
\usepackage{hyperref}
\hypersetup{
  pdfauthor = {Morimichi Kawasaki and Ryuma Orita},
  pdfkeywords = {Symplectic manifolds, groups of Hamiltonian diffeomorphisms, moment maps, symplectic quasi-states, heavy subsets},
  pdftitle = {Rigid fibers of spinning tops},
  pdfpagemode = UseNone,
  bookmarksnumbered=true,
}

\title{Rigid fibers of spinning tops}

\author[1,*]{Morimichi Kawasaki}
\author[2,{\dag}]{Ryuma Orita}
\affil[1]{Research Institute for Mathematical Sciences, Kyoto University, Kyoto 606-8502, Japan}
\affil[2]{Department of Mathematical Sciences, Tokyo Metropolitan University, Tokyo 192-0397, Japan}
\affil[*]{\texttt{kawasaki@kurims.kyoto-u.ac.jp}}
\affil[{{\dag}}]{\texttt{ryuma.orita@gmail.com}}

\date{January 30, Reiwa 2}

\newtheorem{theorem}{Theorem}[section]
\newtheorem{lemma}[theorem]{Lemma}
\newtheorem{proposition}[theorem]{Proposition}
\newtheorem{corollary}[theorem]{Corollary}
\newtheorem{problem}[theorem]{Problem}
\newtheorem{conjecture}[theorem]{Conjecture}

\theoremstyle{definition}
\newtheorem{definition}[theorem]{Definition}
\newtheorem{example}[theorem]{Example}

\theoremstyle{remark}
\newtheorem{remark}[theorem]{Remark}

\newcommand{\Ham}{\mathrm{Ham}}

\newcommand{\supp}{\mathrm{supp}}
\newcommand{\RR}{\mathbb{R}}

\newcommand{\ZZ}{\mathbb{Z}}

\newcommand{\SO}{\mathrm{SO}}
\newcommand{\sgn}{\mathrm{sgn}}

\newcommand{\relmiddle}[1]{\mathrel{}\middle#1\mathrel{}}

\begin{document}

\maketitle

\begin{abstract}
(Non-)displaceability of fibers of integrable systems has been an important problem in symplectic geometry.
In this paper, for a large class of classical Liouville integrable systems containing the Lagrangian top, the Kovalevskaya top and the C. Neumann problem, we find a non-displaceable fiber for each of them.
Moreover, we show that the non-displaceable fiber which we detect is the unique fiber which is non-displaceable from the zero-section.
As a special case of this result, we also show that a singular level set of a convex Hamiltonian is non-displaceable from the zero-section.
To prove these results, we use the notion of superheaviness introduced by Entov and Polterovich.
\end{abstract}

\tableofcontents





\section{Introduction}\label{intro}

\subsection{Backgrounds}

Let $(M,\omega)$ be a symplectic manifold.
A subset $X\subset M$ is called \textit{displaceable} from a subset $Y\subset M$
if there exists a Hamiltonian $H\colon [0,1] \times M\to\RR$ with compact support such that $\varphi_H(X)\cap\overline{Y}=\emptyset$,
where $\varphi_H$ is the \textit{Hamiltonian diffeomorphism} generated by $H$ (see Section \ref{notation} for the definition)
and $\overline{Y}$ is the topological closure of $Y$.
Otherwise, $X$ is called \textit{non-displaceable} from $Y$.
For simplicity, we call $X$ (non-)displaceable if $X$ is (non-)displaceable from $X$ itself.

The problem of (non-)displaceability of a subset of a symplectic manifold (from another subset or from itself) has attracted much attention in symplectic geometry.
Non-displaceability results often pinpoint symplectic rigidity,
namely the difference between symplectic topology and differential topology,
and lead to interesting results in symplectic topology and Hamiltonian dynamics, see for example \cite{PPS}.
In this paper, all symplectic manifolds are cotangent bundles $T^*N$ over closed smooth manifolds $N$,
equipped with the standard symplectic form.
These are the phase spaces of classical mechanics.
The first result on non-displaceability in cotangent bundles was non-displaceability of the zero-section \cite{Gr,LS,Ho,Fl}.
The traditional tools (Morse theory for generating functions, $J$-holomorphic curves, and Floer homology) work only when the set in question is a submanifold.
However, many dynamically relevant subsets of cotangent bundles are not submanifolds.
Examples are energy levels of autonomous Hamiltonians at which the qualitative behavior of the dynamics changes,
like Ma\~n\'e's critical values, and certain subsets therein.
In \cite{EP06}, Entov and Polterovich used Floer homology to construct a function theoretical method
that is designed to detect the non-displaceability of arbitrary closed subsets (we refer to \cite{E,PR} as good surveys).
This theory was adapted by Monzner, Vichery, and Zapolsky \cite{MVZ} to cotangent bundles.
In this paper we use their theory to prove the non-displaceability of fibers of classical integrable systems or the energy level corresponding
to Ma\~n\'e's critical value.

We also note that there are some extrinsic applications of non-displaceability.
Polterovich \cite{P14} proved the existence of a invariant measure of some Hamiltonian flow using non-displaceability of some subset in certain situations.
He \cite{P98} also constructed a Hamiltonian diffeomorphism with arbitrary large Hofer's norm using non-displaceability of $E\times 0_{S^1}$ in $S^2\times T^\ast S^1$, where $E$ is the equator of $S^2$ and $0_{S^1}$ is the zero-section of $T^*S^1$.
In \cite{Ka17}, the first author posed some generalization of Bavard's duality theorem.
Combing it with Polterovich's above result, he pointed out that the existence of stably non-displaceable fibers might be related to the existence of partial quasi-morphisms on the group of Hamiltonian diffeomorphisms.

Many examples of non-displaceable subsets are given as fibers of some Liouville integral systems.
Let $k$ be a positive integer.
We call a smooth map $\Phi=(\Phi_1,\ldots,\Phi_k)\colon M\to\RR^k$ a \textit{moment map} if $\{\Phi_i,\Phi_j\} =0$ for all $1\leq i,j\leq k$,
where $\{\cdot,\cdot\}$ denotes the Poisson bracket on $(M,\omega)$.
A moment map $\Phi=(\Phi_1,\ldots,\Phi_k)\colon M\to\RR^k$ is called a \textit{Liouville integrable system} if $k=\dim{M}/2$ and $(d\Phi_1)_x,\ldots,(d\Phi_k)_x$ are linearly independent almost everywhere.

Especially, many researchers have studied (non-)displaceable fibers of Liouville integrable system associated with toric structures.
For example, see \cite{BEP,Ch,EP09,Mc,FOOO10,FOOO11a,FOOO12,AbMa,ABM,KLS,AFOOO}.
Recently some researchers study (non-)displaceable fibers of ``moment maps" associated with various generalizations of toric structure like Gelfand--Cetlin systems, semi-toric structures and so on (see, e.g., \cite{NNU,W,Vi,CKO,KO19b}).

In this paper, we deal with classical Liouville integrable systems on contangent bundles.
We study (non-)displaceable fibers of moment maps on the cotangent bundle of the two-sphere $S^2$ or the three-dimensional rotation group $\SO(3)$ which appear in classical mechanics, for example, the spherical pendulum, the Lagrange top and the Kovalevskaya top.
As a previous research in a similar direction, we refer to Albers--Frauenfelder's work \cite{AF08}.
They proved non-displaceability of the Polterovich torus in $T^\ast S^2$ which can be regarded as a fiber of some Liouville integrable system.

As a general fact on (non-)displaceability of fibers of moment maps,
Entov and Polterovich \cite{EP06} proved the following theorem.

\begin{theorem}[{\cite[Theorem 2.1]{EP06}}]\label{existence of non-disp fiber}
Let $(M,\omega)$ be a closed symplectic manifold
and $\Phi=(\Phi_1,\ldots,\Phi_k)\colon M\to\RR^k$ a moment map.
Then, there exists $y_0\in\Phi(M)$ such that $\Phi^{-1}(y_0)$ is non-displaceable.
\end{theorem}

To prove Theorem \ref{existence of non-disp fiber}, Entov and Polterovich \cite{EP06} introduced the concept of partial symplectic quasi-state (see Definition \ref{def:psqs}).
In \cite{EP09}, they introduced the notion of heaviness of closed subsets in terms of partial symplectic quasi-states.
Let $C_c(M)$ denote the set of continuous functions on $M$ with compact supports.

\begin{definition}[{\cite[Definition 1.3]{EP09}}]\label{def:hv}
Let $\zeta\colon C_c(M)\to\RR$ be a partial symplectic quasi-state on $(M,\omega)$.
A compact subset $X$ of $M$ is said to be \textit{$\zeta$-heavy} (resp.\ \textit{$\zeta$-superheavy}) if
\[
	\zeta(H)\geq\inf_X H \quad \left(\text{resp.}\ \zeta(H)\leq\sup_X H\right)
\]
for any $H\in C_c(M)$.
\end{definition}

Here we collect properties of (super)heavy subsets.

\begin{theorem}[{\cite[Theorem 1.4]{EP09}}]\label{prop:shv is hv}
Let $\zeta\colon C_c(M)\to\RR$ be a partial symplectic quasi-state on $(M,\omega)$.
\begin{enumerate}
\item Every $\zeta$-superheavy subset is $\zeta$-heavy.
\item Every $\zeta$-heavy subset is non-displaceable.
\item Every $\zeta$-heavy subset is non-displaceable from every $\zeta$-superheavy subset.
In particular, every $\zeta$-heavy subset intersects every $\zeta$-superheavy subset.
\end{enumerate}
\end{theorem}

\subsection{Main results}\label{sec:main results}

In this paper we prove that some classical integrable systems (e.g., Lagrange top and Kovalevskaya top) admit superheavy fibers.
We consider the cotangent bundle $(T^*N,\omega_0)$ of a closed smooth $n$-dimensional manifold $N$
where $\omega_0$ is the standard symplectic form on $T^*N$.
Let $(q,p)$ be canonical coordinates on $T^*N$ where $q\in N$ and $p\in T^*_qN$.
Let $\pi\colon T^*N\to N$ denote the natural projection.

\begin{definition}\label{star}
A (time-independent) Hamiltonian $H\colon T^*N\to\RR$ satisfies \textit{condition $(\star)$} if the following conditions hold.
\begin{enumerate}
	\item For any $c\in\RR$ the sublevel set $H^{-1}\bigl((-\infty,c]\bigr)\subset T^*N$ is compact.
	\item For any $q\in N$, \[H(q,0)=\min_{p\in T_q^*N}H(q,p).\]
\end{enumerate}
\end{definition}

For a Hamiltonian $H\colon T^*N\to\RR$ satisfying condition $(\star)$, we set
\begin{equation}\label{eq:m_H}
	m_H=\max_{q\in N}\min_{p\in T_q^*N}H(q,p)\quad\text{and}\quad%
	S_H=H^{-1}(m_H)\cap 0_N.
\end{equation}

Typical examples of Hamiltonians satisfying condition $(\star)$ are \textit{convex Hamiltonians}
\begin{equation}\label{eq:convex}
	H(q,p)=\frac{1}{2}\| p\|_g^2+U(q),
\end{equation}
where $\|\cdot\|_g$ is the dual norm of a Riemannian metric $g$ on $N$
and $U\colon N\to\RR$ is a smooth potential.
In this case, the value $m_H$ equals the \textit{Ma\~n\'e critical value} $\max_N U$
and
\[
	S_H=\left\{\,(q,0) \in T^*N\relmiddle| U(q) = \max_NU \,\right\}.
\]

In Section \ref{secexam}, we provide classical examples satisfying the assumption of Theorem \ref{main theorem}.

To prove non-displaceability of a fiber of some integrable systems, we use the following partial symplectic quasi-state.
In \cite{Oh97,Oh99}, Oh constructed a spectral invariant on $(T^*N,\omega_0)$ in terms of the Lagrangian Floer theory of the zero-section $0_N$ of $T^*N$.
In \cite{MVZ}, Monzner, Vichery, and Zapolsky constructed a partial symplectic quasi-state on $(T^*N,\omega_0)$,
denoted by $\zeta_{\mathrm{MVZ}}\colon C_c(T^*N)\to\RR$, as the asymptotization of Oh's Lagrangian spectral invariant.
In this paper, the following property of $\zeta_{\mathrm{MVZ}}$ is crucial.

\begin{proposition}[{\cite[Example 1.19]{MVZ}}]\label{zero-sec}
The zero-section $0_N\subset T^*N$ is $\zeta_{\mathrm{MVZ}}$-superheavy.
\end{proposition}

Now we are in a position to state the main result of this paper.

\begin{theorem}\label{main theorem}
Let $N$ be a closed manifold
and $\Phi=(\Phi_1,\ldots,\Phi_k)\colon T^*N\to\RR^k$ a moment map.
Assume that $\Phi_1$ satisfies condition $(\star)$ and that the set $\Phi(S_{\Phi_1})$ is a singleton, i.e.,
$\Phi(S_{\Phi_1})=\{y_0\}$ for some $y_0\in\RR^k$.
Then, the fiber $\Phi^{-1}(y_0)$ of $\Phi$ is $\zeta_{\mathrm{MVZ}}$-superheavy.
\end{theorem}



By Theorem \ref{prop:shv is hv} and Proposition \ref{zero-sec}, the fiber $\Phi^{-1}(y_0)$ is non-displaceable from itself and from the zero-section $0_N$.
Moreover, we can prove that every fiber of $\Phi$, other than $\Phi^{-1}(y_0)$, is displaceable from $0_N$.
To refine Theorem \ref{main theorem}, we introduce the notion of \textit{$X$-stems}.

\begin{definition}[\cite{Ka2}]\label{def:X-stem}
Let $(M,\omega)$ be a symplectic manifold and $X$ a compact subset of $M$.
A compact subset $Y$ of $M$ is called an \textit{$X$-stem}
if there exists a moment map $\Phi=(\Phi_1,\ldots,\Phi_k)\colon M\to\RR^k$ satisfying the following conditions:
\begin{enumerate}
\item $Y=\Phi^{-1}(y_0)$ for some $y_0\in\Phi(M)$.
\item Every fiber of $\Phi$, other than $Y$, is displaceable from itself or from $X$.
\end{enumerate}
\end{definition}

Entov and Polterovich \cite{EP06} introduced the notion of \textit{stems} (i.e., every fiber of $\Phi$, other than $\Phi^{-1}(y_0)$, is displaceable, where $\Phi\colon T^*N\to\RR^k$ is a moment map) and proved that stems are superheavy with respect to any partial symplectic quasi-state \cite[Theorem 1.8]{EP09}.
We note that every stem is an $X$-stem for any compact subset $X$.
We have the following result on $X$-stems.

\begin{theorem} \label{theorem:X-stem is shv}
Let $(M,\omega)$ be a symplectic manifold,
$\zeta\colon C_c(M)\to\RR$ a partial symplectic quasi-state on $(M,\omega)$,
and $X$ a $\zeta$-superheavy subset of $M$.
Then every $X$-stem is $\zeta$-superheavy.
\end{theorem}

The proof of Theorem \ref{theorem:X-stem is shv} is similar to that of \cite[Theorem 2.5]{KO19b}.
Theorem \ref{theorem:X-stem is shv} refines Theorem \ref{main theorem} as follows.

\begin{theorem}\label{main theorem2}
Let $N$ be a closed manifold
and $\Phi=(\Phi_1,\ldots,\Phi_k)\colon T^*N\to\RR^k$ a moment map.
Assume that $\Phi_1$ satisfies condition $(\star)$ and that the set $\Phi(S_{\Phi_1})$ is a singleton, i.e.,
$\Phi(S_{\Phi_1})=\{y_0\}$ for some $y_0\in\RR^k$.
Then, every fiber of $\Phi$, other than $\Phi^{-1}(y_0)$, is displaceable from the zero-section $0_N$.
In particular, the fiber $\Phi^{-1}(y_0)$ is a $0_N$-stem.
Hence, by Theorem \ref{theorem:X-stem is shv} and Proposition \ref{zero-sec}, $\Phi^{-1}(y_0)$ is $\zeta_{\mathrm{MVZ}}$-superheavy.
\end{theorem}

We prove Theorem \ref{main theorem2} in Section \ref{proofmainthm}.
By Theorem \ref{prop:shv is hv}, we see that $\Phi^{-1}(y_0)$ is the unique fiber which is non-displaceable from $0_N$.
On the other hand, it is a natural question to ask whether $\Phi^{-1}(y_0)$ is a stem.
In Conjecture \ref{conjecture mugen}, the authors expect that $\Phi\colon T^*N\to\RR^k$ has infinitely many non-displaceable fibers, in particular, $\Phi^{-1}(y_0)$ is not a stem in a more general situation.
For evidences supporting Conjecture \ref{conjecture mugen}, see Section \ref{sec:Rab}.



Here we provide two other applications of our arguments.

\begin{theorem}\label{secondcorollary}
Let $H_1,\ldots,H_k\colon T^*N\to\RR$ be Hamiltonians satisfying condition $(\star)$ and $\{H_i,H_j\}=0$ for all $1\leq i,j\leq k$.
Then, $\bigcap_{i=1}^k S_{H_i}\neq\emptyset$.
\end{theorem}

For example, the functions $H$ and $G$ in Example \ref{neu} (C. Neumann problem) satisfy condition $(\star)$ and one can confirm that $S_H\cap S_G\neq\emptyset$.
As another example, the functions $H$ and $G$ in Example \ref{Clebsch} (Clebsch top) also satisfy condition $(\star)$ and we have $S_H\cap S_G\neq\emptyset$.
We prove Theorem \ref{secondcorollary} in Section \ref{proofcor}.
The authors do not know another proof of this misterious theorem without using the Floer theory.

\begin{proposition}\label{mainprop}
Let $\Phi=(\Phi_1,\ldots,\Phi_k)\colon T^*N\to\RR^k$ be a moment map.
Assume that $\Phi_1$ satisfies condition $(\star)$ and that the set $\Phi(S_{\Phi_1})$ is a singleton, i.e.,
$\Phi(S_{\Phi_1})=\{y_0\}$ for some $y_0\in\RR^k$.
Then, $\pi\bigl(\Phi^{-1}(y_0)\bigr)=N$.
\end{proposition}

When $k=1$, the proof of Proposition \ref{mainprop} is straightforward by the definition of $m_{\Phi_1}$.
Proposition \ref{mainprop} follows immediately from Theorem \ref{main theorem} and the following proposition.

\begin{proposition}\label{maintopo}
If $X$ is a $\zeta_{\mathrm{MVZ}}$-superheavy subset of $T^*N$, then $\pi(X)=N$.
\end{proposition}

We prove Proposition \ref{maintopo} in Section \ref{proofprop}.


\section{Applications}

In this section, we deal with some classical integrable systems satisfying the assumption of Theorem \ref{main theorem} and detect superheavy fibers of them.

\subsection{Relationship with Ma\~n\'e's critical values}\label{sec:Mane}
We provide an application of our main theorem when a moment map is a function.

Let $(N,g)$ be a closed Riemannian manifold.
We equip the cotangent bundle $T^*N$ with the standard symplectic form $\omega_0$.
In the context of Ma\~n\'e's critical values, Cieliebak, Frauenfelder, and Paternain \cite{CFP} proved the following theorem.

\begin{theorem}[{\cite[Theorem 1.2]{CFP}}]\label{thm:CFP}
Let $(N,g)$ be a closed Riemannian manifold
and $H\colon T^*N\to\RR$ a convex Hamiltonian $($see \eqref{eq:convex} for the definition$)$.
Then, the level set $H^{-1}(m_H)\subset T^*N$ is non-displaceable.
\end{theorem}

As a corollary of our main theorem (Theorem \ref{main theorem}),
we can prove that the level set $H^{-1}(m_H)$ in Theorem \ref{thm:CFP} is non-displaceable from the zero-section $0_N$ in a more general setting.

\begin{corollary}\label{cor:CFP}
Let $N$ be a closed manifold
and $H\colon T^*N\to\RR$ a Hamiltonian satisfying condition $(\star)$.
Then, the level set $H^{-1}(m_H)\subset T^*N$ is non-displaceable from itself and from the zero-section $0_N$.
\end{corollary}

\begin{remark}\label{non-disp regular CFP}
Actually, Cieliebak, Frauenfelder, and Paternain \cite{CFP} proved the non-displaceability of $H^{-1}(c)$ for any $c>m_H$ using the Rabinowitz Floer theory.
Hence they obtained Theorem \ref{thm:CFP} as its corollary.
On the other hand, as stated in Theorem \ref{main theorem2}, $H^{-1}(c)$ is displaceable from $0_N$ for any $c\neq m_H$.
\end{remark}

\begin{example}[Pendulum]\label{pendulum}
The pendulum is the Hamiltonian system with one degree of freedom on the cotangent bundle $T^*S^1$ of the unit circle $S^1=\RR/2\pi\ZZ$.
We define a function $H\colon T^*S^1\to\RR$ by
\[
	H(q,p)=\frac{1}{2}p^2+(1-\cos q).
\]
Then, $H$ satisfies condition $(\star)$ and
\[
	m_H=\max_{q\in S^1}{(1-\cos q)}=2.
\]
By Theorem \ref{main theorem},
the level set $H^{-1}(2)\subset T^*S^1$ is $\zeta_{\mathrm{MVZ}}$-superheavy.
$H^{-1}(2)$ is homeomorphic to the figure eight.
Note that the $\zeta_{\mathrm{MVZ}}$-superheaviness of $H^{-1}(2)$ also follows from \cite[Proposition 1.22]{MVZ}.
\end{example}


\subsection{Classical integrable systems}\label{secexam}

\begin{example}[Spherical pendulum]\label{ex:sph pen}
The spherical pendulum \cite{La} describes a motion of a particle moving on the unit two-sphere
\begin{equation}\label{eq:S2}
	S^2=\left\{\, q=(q_1,q_2,q_3)\in\RR^3\relmiddle| q_1^2+q_2^2+q_3^2=1\,\right\}
\end{equation}
under a gravitational force.
Let $g_0$ denote the standard Riemannian metric on $S^2$.
We define functions $\underline{H},\underline{G}\colon TS^2\to\RR$ by
\[
	\underline{H}(q,v)=\frac{1}{2}\| v\|_{g_0}^2+q_3\quad\text{and}\quad%
	\underline{G}(q,v)=q_1v_2-q_2v_1
\]
for $(q,v)=(q_1,q_2,q_3,v_1,v_2,v_3)\in TS^2\subset T\RR^3\cong\RR^3\times\RR^3$, respectively.
Let $\Psi\colon TS^2\to T^*S^2$ denote the Legendre transformation of $\underline{H}$.
We then define functions on $T^*S^2$ by $H=\underline{H}\circ\Psi^{-1}$ and $G=\underline{G}\circ\Psi^{-1}$.
Then, $\{H,G\}=0$ and the function $H$ satisfies condition $(\star)$.
We set $\Phi=(H,G)\colon T^*S^2\to\RR^2$.
Since $S_H=\{(0,0,1,0,0,0)\}$, we have $\Phi(S_H)=\{(1,0)\}$.
By Theorem \ref{main theorem},
the fiber $\Phi^{-1}(1,0)\subset T^*S^2$ is $\zeta_{\mathrm{MVZ}}$-superheavy.
In particular, $\Phi^{-1}(1,0)$ is non-displaceable from itself and from $0_{S^2}$.
We note that the value $(1,0)$ corresponds to the focus-focus singularity of this system
and the fiber $\Phi^{-1}(1,0)$ is homeomorphic to the two-dimensional torus pinched at a single point (see \cite[Section I\hspace{-.1em}V.3.4]{CB}).
\end{example}

\begin{remark}\label{Pol torus}
Brendel, Kim, and Schlenk \cite{BKS} proved that the fiber $\Phi^{-1}(c,0)$ is non-displaceable for any $c>1$.
Thus, the non-displaceability of $\Phi^{-1}(1,0)$ immediately follows.
On the other hand, as stated in Theorem \ref{main theorem2}, $\Phi^{-1}(c,0)$ is displaceable from $0_{S^2}$ for any $c>1$.

The authors do not know whether there exist a Hamiltonian $H$ satisfying condition ($\star$) and a real number $c$ with $c>m_H$ such that $H^{-1}(c)$ is displaceable.
\end{remark}

\begin{example}[C. Neumann problem]\label{neu}

Let $a_1,a_2,a_3$ be positive numbers satisfying $a_1<a_2<a_3$.
Let $S^2\subset\RR^3$ denote the unit two-sphere as in \eqref{eq:S2}.
In \cite{Neu}, C. Neumann introduced a Hamiltonian system on $T^*S^2$
which describes the motion of a particle on the unit two-sphere $S^2$
under the influence of the linear force $-(a_1q_1,a_2q_2,a_3q_3)$.
We define functions $\underline{H},\underline{G}\colon TS^2\to\RR$ by
\[
	\underline{H}(q,v)=\frac{1}{2}\|v\|_{g_0}^2+\frac{1}{2}(a_1q_1^2+a_2q_2^2+a_3q_3^2)
\]
and
\[
	\underline{G}(q,v)=\frac{1}{2}\sum_{i=1}^3 a_iv_i^2+\frac{1}{2}\|v\|_{g_0}^2\sum_{i=1}^3 a_iq_i^2+\frac{1}{2}\sum_{i=1}^3 a_i^2 q_i^2
\]
for $(q,v)=(q_1,q_2,q_3,v_1,v_2,v_3)\in TS^2\subset T\RR^3\cong\RR^3\times\RR^3$, respectively.
Let $\Psi\colon TS^2\to T^*S^2$ denote the Legendre transformation of $\underline{H}$.
We then define functions $H, G\colon T^*S^2\to\RR$ by $H=\underline{H}\circ\Psi^{-1}$ and $G=\underline{G}\circ\Psi^{-1}$.
Then, $\{H,G\}=0$ and the function $H$ satisfies condition $(\star)$.
We set $\Phi=(H,G)\colon T^*S^2\to\RR^2$.
Since $S_H=\{(0,0,\pm 1,0,0,0)\}$, we have $\Phi(S_H)=\{(a_3/2,a_3^2/2)\}$.
By Theorem \ref{main theorem},
the fiber $\Phi^{-1}(a_3/2,a_3^2/2)\subset T^*S^2$ is $\zeta_{\mathrm{MVZ}}$-superheavy.
\end{example}


\subsubsection{Spinning tops}\label{sec:tops}

We consider the motion of tops.
Let $q_1\cdot q_2$ (resp.\ $q_1\times q_2$) denote the dot (resp.\ cross) product of $q_1$ and $q_2$ in $\RR^3$.
Let
\[
	\SO(3)=\left\{\,(q_1,q_2,q_3)\in\mathrm{M}_3(\RR)\relmiddle| q_1,q_2,q_3\in S^2,\ q_1\cdot q_2=0,\ q_3=q_1\times q_2\,\right\}
\]
denote the three-dimensional rotation group,
where $S^2\subset\RR^3$ is the unit two-sphere as in \eqref{eq:S2}.
Let $(e_1,e_2,e_3)$ denote the identity matrix.
Given a point $(q_1,q_2,q_3)\in\SO(3)$, we set $n_i=q_i\cdot e_3$ for each $i=1,2,3$.

Let $(q,\omega)=(q_1,q_2,q_3,\omega_1,\omega_2,\omega_3)$ be the canonical coordinates on the tangent bundle $T\SO(3)$ defined in terms of the angular velocity (see, for example, \cite[Section 26]{Ar}).
Let $0_{\SO(3)}$ denote the zero-section of $T^*\SO(3)$.

Let $I_1$, $I_2$, $I_3$ be positive numbers and $f\colon [-1,1]^3\to\RR$ a smooth function.
We define functions $\underline{H},\underline{L_z}\colon T\SO(3)\to\RR$ by
\begin{equation}\label{eq:top H}
	\underline{H}(q,\omega)=\frac{1}{2}(I_1\omega_1^2+I_2\omega_2^2+I_3\omega_3^2)+(f\circ\nu)(q_1,q_2,q_3)
\end{equation}
and
\begin{equation}\label{eq:top L}
	\underline{L_z}(q,\omega)=I_1n_1\omega_1+I_2n_2\omega_2+I_3n_3\omega_3,
\end{equation}
respectively,
where $\nu\colon\SO(3)\to [-1,1]^3$ is the map defined by $\nu(q_1,q_2,q_3)=(n_1,n_2,n_3)$.

Let $\Psi\colon T\SO(3)\to T^*\SO(3)$ denote the Legendre transformation of $\underline{H}$.
Note that $\Psi\colon T\SO(3)\to T^*\SO(3)$ is the metric dual operation with respect to the Riemannian metric $g$ on $\SO(3)$
defined by
\[
 g_q(\omega,\omega')=I_1\omega_1\omega'_1+I_2\omega_2\omega'_2+I_3\omega_3\omega'_3
\]
for $q\in\SO(3)$ and $\omega=(\omega_1,\omega_2,\omega_3)$, $\omega'=(\omega'_1,\omega'_2,\omega'_3)\in T_q\SO(3)$.

We then define functions on $T^*\SO(3)$ by $H=\underline{H}\circ\Psi^{-1}$ and $L_z=\underline{L_z}\circ\Psi^{-1}$.
Then, $\{H,L_z\}=0$ and the function $H$ satisfies condition $(\star)$.
We note that
\begin{equation}\label{eq:top SH}
	S_H=\left\{\,(q,0)\in T^*\SO(3)\relmiddle| H(q,0)=\max_{\SO(3)}f\circ\nu\,\right\}.
\end{equation}
Hence $L_z(S_H)\subset L_z(0_{\SO(3)})=\{0\}$.


\begin{example}\label{general top}
We set $\Phi=(H,L_z)\colon T^*\SO(3)\to\RR^2$.
Then,
\[
	\Phi(S_H)=\left\{\left(\max_{\SO(3)}{f\circ\nu},0\right)\right\}.
\]
By Theorem \ref{main theorem}, the fiber $\Phi^{-1}(\max_{\SO(3)}{f\circ\nu},0)$ is $\zeta_{\mathrm{MVZ}}$-superheavy.
\end{example}

\begin{example}[Lagrange top]\label{rigidlag}
The \textit{Lagrange top} \cite{La,Ar} is a top such that $I_1=I_2$ and $f(x,y,z)=cz$ for some real number $c$.
We define another function $\underline{G}\colon T\SO(3)\to\RR$ by
\[
	\underline{G}(q,\omega)=I_3\omega_3,
\]
and set $G=\underline{G}\circ\Psi^{-1}$.
Then, $\{H,G\}=0$ and $\{L_z,G\}=0$.
We set $\Phi=(H,L_z,G)\colon T^*\SO(3)\to\RR^3$.
By \eqref{eq:top SH}, $H(S_H)=\{\lvert c\rvert\}$ and $G(S_H)=\{0\}$.
Therefore, $\Phi(S_H)=\{(\lvert c\rvert,0,0)\}$.
By Theorem \ref{main theorem}, the fiber $\Phi^{-1}(\lvert c\rvert,0,0)$ is $\zeta_{\mathrm{MVZ}}$-superheavy.
If $|c|\neq0$, the fiber $\Phi^{-1}(\lvert c\rvert,0,0)$ is homeomorphic to a 3-torus with a normal crossing along an $S^1$.
For more precise description of this fiber and its singularity, see \cite[Section V.6]{CB}.
\end{example}

\begin{example}[Kovalevskaya top]
The \textit{Kovalevskaya top} \cite{Ko} is a top such that $I_1=I_2=2I_3$ and $f(x,y,z)=ax$ for some real number $a$.
We define another function $\underline{G}\colon T\SO(3)\to\RR$ by
\[
	\underline{G}(q,\omega)=\left(\omega_1^2-\omega_2^2-\frac{2a}{I_1}n_1\right)^2+\left(2\omega_1\omega_2-\frac{2a}{I_1}n_2\right)^2,
\]
and set $G=\underline{G}\circ\Psi^{-1}$.
Then, $\{H,G\}=0$ and $\{L_z,G\}=0$.
We set $\Phi=(H,L_z,G)\colon T^*\SO(3)\to\RR^3$.
By \eqref{eq:top SH}, $H(S_H)=\{\lvert a\rvert\}$.
If $a\neq 0$, then
\[
	S_H=\{\,(q_1,q_2,q_3,0,0,0)\in T^*\SO(3)\mid q_1=\sgn(a)e_3\,\},
\]
where $\sgn(a)$ is the signature of $a$,
and hence $G(S_H)=\{4a^2/I_1^2\}$.
If $a=0$, then $S_H=0_{\SO(3)}$,
and hence $G(S_H)=\{0\}$.

Therefore, given $a\in\RR$,
we have $\Phi(S_H)=\{(\lvert a\rvert,0,4a^2/I_1^2)\}$.
By Theorem \ref{main theorem}, the fiber $\Phi^{-1}(\lvert a\rvert,0,4a^2/I_1^2)$ is $\zeta_{\mathrm{MVZ}}$-superheavy.
\end{example}

\begin{example}[Clebsch top]\label{Clebsch}
The \textit{Clebsch top} \cite{Cl} is a top such that
$I_1<I_2<I_3$ and
\[
	f(x,y,z)=\frac{1}{2I_1I_2I_3}(I_1x^2+I_2y^2+I_3z^2).
\]
This system describes a motion of a rigid body, fixed in its center of gravity, in an ideal fluid.
We define another function $\underline{G}\colon T\SO(3)\to\RR$ by
\[
	\underline{G}(q,\omega)=\frac{1}{2}(I_1^2\omega_1^2+I_2^2\omega_2^2+I_3^2\omega_3^2)-\frac{1}{2I_1I_2I_3}(I_2I_3n_1^2+I_3I_1n_2^2+I_1I_2n_3^2),
\]
and set $G=\underline{G}\circ\Psi^{-1}$.
Then, $\{H,G\}=0$ and $\{L_z,G\}=0$.
We set $\Phi=(H,L_z,G)\colon T^*\SO(3)\to\RR^3$.
Since $I_1<I_2<I_3$, by \eqref{eq:top SH},
\[
	S_H=\{\,(q_1,q_2,q_3,0,0,0)\in T^*\SO(3)\mid n_3=\pm 1\,\}.
\]
Then,
\[
	\Phi(S_H)=\left\{\left(\frac{1}{2I_1I_2},0,-\frac{1}{2I_3}\right)\right\}.
\]
By Theorem \ref{main theorem}, the fiber $\Phi^{-1}\bigl((2I_1I_2)^{-1},0,-(2I_3)^{-1}\bigr)$ is $\zeta_{\mathrm{MVZ}}$-superheavy.
\end{example}

\begin{remark}
We can also apply our main theorem to other famous Liouville integrable systems such as the Euler top \cite{Eu,Ar}.
However, the corresponding $\zeta_{\mathrm{MVZ}}$-superheavy fiber of the Euler top contains the zero-section which is already known to be $\zeta_{\mathrm{MVZ}}$-superheavy.
In this sense, our theorem gives only trivial results for such examples.
\end{remark}

\subsection{On the existence of infinitely many non-displaceable fibers}\label{sec:Rab}

It is a natural question to ask whether a Liouville integrable system has infinitely many non-displaceable fibers.
Along this line, we have the following result.

Let $(N,g)$ be a closed Riemannian manifold.
Given a positive number $r$, let
\[
	S^*_{g,r}N=\{\,(q,p)\in T^*N\mid\,\| p\|_g=r\}
	\quad\text{and}\quad B^*_{g,r}N=\{\,(q,p)\in T^*N\mid\,\| p\|_g<r\}
\]
denote the sphere subbundle of radius $r$ and the open ball subbundle of radius $r$, respectively.

\begin{proposition}\label{rabi mugen}
Let $(N,g)$ be a closed Riemannian manifold.
Assume that for any positive number $r$ there exist a positive number $R$ with $R>r$
and a partial symplectic quasi-state $\zeta_R\colon C_c(T^\ast N)\to\mathbb{R}$ such that $S^\ast_{g,R}N$ is $\zeta_R$-superheavy.
Let $H\colon T^*N\to\RR$ be a Hamiltonian such that $H^{-1}\bigl((-\infty,c]\bigr)$ is compact for any $c\in\mathbb{R}$.
Then, every moment map $\Phi=(\Phi_1,\ldots,\Phi_k)\colon T^*N\to\RR^k$ with $\Phi_1=H$ has infinitely many non-displaceable fibers.
\end{proposition}

We prove Proposition \ref{rabi mugen} in Section \ref{Rabi section}.

The authors do not know examples of Riemannian manifolds satisfying the assumption of Proposition \ref{rabi mugen}.
However, the authors expect that every closed Riemannian manifold satisfies the assumption due to the following reason.
Given a Riemannian metric $g$ on $N$ and a positive number $R$, it is known that the Rabinowitz Floer homology of $S^\ast_{g,R}N$ is non-trivial \cite{CFO}.
Thus, one can construct a Rabinowitz spectral invariant (with respect to the fundamental class) from the Rabinowitz Floer homology through Albers--Fauenfelder's construction \cite{AF}.
We expect that the asymptotization $\zeta$ of that spectral invariant is a partial symplectic quasi-state and $S^\ast_{g,R}N$ is $\zeta$-superheavy since $\zeta$ is constructed from the Rabinowitz Floer theory of $S^\ast_{g,R}N$.

By Proposition \ref{rabi mugen} and the above expectation, we pose the following conjecture.

\begin{conjecture}\label{conjecture mugen}
Let $N$ be a closed manifold.
Let $H\colon T^*N\to\RR$ be a Hamiltonian such that $H^{-1}\bigl((-\infty,c]\bigr)$ is compact for any $c\in\mathbb{R}$.
Then, every moment map $\Phi=(\Phi_1,\ldots,\Phi_k)\colon T^*N\to\RR^k$ with $\Phi_1=H$ has infinitely many non-displaceable fibers.
\end{conjecture}

Actually, this conjecture is true when $\Phi$ is the spherical pendulum (Remark \ref{Pol torus}) or a convex Hamiltonian (Remark \ref{non-disp regular CFP}).


\section{Preliminaries}\label{prelim}

In this section, we first set conventions and notation.
Then we define partial symplectic quasi-states.
Let $(M,\omega)$ be a symplectic manifold.

\subsection{Conventions and notation}\label{notation}

Let $H$ be a one-periodic in time Hamiltonian with compact support,
i.e., a smooth function $H\colon [0,1] \times M\to\RR$ with compact support.
We set $H_t=H(t,\cdot)$ for $t\in [0,1]$.
The \textit{Hamiltonian vector field} $X_{H_t}\in \mathfrak{X}(M)$ associated to $H_t$ is defined by
\[
	\iota_{X_{H_t}}\omega=-dH_t.
\]
The \textit{Hamiltonian isotopy} $\{\varphi_H^t\}_{t\in\RR}$ associated to $H$ is defined by
\[
	\begin{cases}
		\varphi_H^0=\mathrm{id},\\
		\frac{d}{dt}\varphi_H^t=X_{H_t}\circ\varphi_H^t\quad \text{for all}\ t\in\RR,
	\end{cases}
\]
and its time-one map $\varphi_H=\varphi_H^1$ is referred to as the \textit{Hamiltonian diffeomorphism with compact support} generated by $H$.
Let $\Ham(M)$ denote the group of Hamiltonian diffeomorphisms of $M$ with compact supports.


\subsection{Partial symplectic quasi-states}\label{psqs}

Let $C_c^\infty(M)$ denote the set of smooth functions on $M$ with compact supports.

\begin{definition}[{\cite{EP06,FOOO11b,PR,KO19b}}]\label{def:psqs}
A \textit{partial symplectic quasi-state} on $(M,\omega)$ is a functional $\zeta\colon C_c(M)\to\RR$ satisfying the following conditions.
\begin{description}
	\item[Normalization] There exists a non-empty compact subset $K_{\zeta}$ of $M$ such that $\zeta(F)=a$ for any real number $a$ and any function $F\in C_c(M)$ with $F|_{K_{\zeta}}\equiv a$.
	\item[Stability] For any $H_1,H_2\in C_c(M)$, we have
	\[
		\min_M(H_1-H_2)\leq\zeta(H_1)-\zeta(H_2)\leq\max_M(H_1-H_2).
	\]
	In particular, \textbf{Monotonicity} holds: $\zeta(H_1)\leq\zeta(H_2)$ if $H_1\leq H_2$.
	\item[Semi-homogeneity] $\zeta(sH)=s\zeta(H)$ for any $H\in C_c(M)$ and any $s>0$.
	\item[Hamiltonian Invariance] $\zeta(H\circ\phi)=\zeta(H)$ for any $H\in C_c(M)$ and any $\phi\in\Ham(M)$.
	\item[Vanishing] $\zeta(H)=0$ for any $H\in C_c(M)$ whose support is displaceable.
	\item[Quasi-subadditivity] $\zeta(H_1+H_2)\leq\zeta(H_1)+\zeta(H_2)$ for any $H_1,H_2\in C_c^\infty(M)$ satisfying $\{H_1,H_2\}=0$.
\end{description}
\end{definition}

\begin{remark}
There are different definitions of partial symplectic quasi-state.
Our definition is based on \cite{KO19b}, but our definition is slightly different from that one.
In \cite{KO19b}, they consider the different normalization condition $\zeta(a)=a$ for every real number $a$.
In this paper, since we consider open symplectic manifolds and functions with compact supports, we cannot define $\zeta(a)$ unless $a=0$.
This is why we take a slightly different normalization condition.
One can easily prove that our definition and the original one are equivalent when $M$ is closed.
\end{remark}

We obtain the following corollary of Theorem \ref{theorem:X-stem is shv} which is an analogue of the main result in \cite{KO19b}.

\begin{corollary}\label{pshv open}
Let $(M,\omega)$ be a symplectic manifold.
Let $\zeta\colon C_c(M)\to\RR$ be a partial symplectic quasi-state on $(M,\omega)$, and $X$ a $\zeta$-superheavy subset of $M$.
Let $H\colon M\to\RR$ be a Hamiltonian such that $H^{-1}\bigl((-\infty,c]\bigr)$ is compact for any $c\in\mathbb{R}$.
Then, every moment map $\Phi=(\Phi_1,\ldots,\Phi_k)\colon M\to\RR^k$ with $\Phi_1=H$ has a fiber that is non-displaceable from itself and from $X$.
\end{corollary}

\begin{proof}
Arguing by contradiction, assume that every fiber of $\Phi$ is displaceable from itself or from $X$.
By the assumption on $H$, every fiber of $\Phi$ is compact.
Then, every fiber is an $X$-stem.
Since $X$ is $\zeta$-superheavy, by Theorem \ref{theorem:X-stem is shv}, every fiber is $\zeta$-superheavy.
Since all fibers are mutually disjoint, it contradicts Theorem \ref{prop:shv is hv} (i) and (iii).
\end{proof}


\section{Proofs of the main results}\label{prooflemma}

In this section, we prove the main results stated in Section \ref{sec:main results}.
Let $N$ be a closed manifold.
Let $\pi\colon T^*N\to N$ denote the natural projection.
We equip $T^*N$ with the standard symplectic form $\omega_0$.

\subsection{Proof of Theorem \ref{main theorem2}}\label{proofmainthm}

For the sake of applications in Sections \ref{sec:Rab} and \ref{Rabi section}, we generalize condition $(\star)$ as follows.

\begin{definition}\label{star_s}
Let $\Sigma$ be a compact subset of $T^*N$.
A (time-independent) Hamiltonian $H\colon T^*N\to\RR$ satisfies \textit{condition $(\star)_{\Sigma}$} if the following conditions hold.
\begin{enumerate}
	\item For any $c\in\RR$ the sublevel set $H^{-1}\bigl((-\infty,c]\bigr)\subset T^*N$ is compact.
	\item For any $q\in N$, \[H|_{T^*_qN\cap\Sigma}\equiv\min_{p\in T_q^*N}H(q,p).\]
\end{enumerate}
\end{definition}

We note that condition $(\star)_{\Sigma}$ is equivalent to condition $(\star)$ when $\Sigma=0_N$.
For a Hamiltonian $H\colon T^*N\to\RR$ satisfying condition $(\star)_{\Sigma}$, we set
\[
	m_H=\max_{q\in N}\min_{p\in T_q^*N}H(q,p)\quad\text{and}\quad%
	S_H^{\Sigma}=H^{-1}(m_H)\cap\Sigma.
\]
Then, $S_H^{0_N}=S_H$ (see \eqref{eq:m_H}).

In this section, we prove the following theorem which generalizes Theorem \ref{main theorem2}.

\begin{theorem}\label{thm:gen}
Let $N$ be a closed manifold, $\Sigma$ a compact subset of $T^*N$,
and $\Phi=(\Phi_1,\ldots,\Phi_k)\colon T^*N\to\RR^k$ a moment map.
Assume that $\Phi_1$ satisfies condition $(\star)_{\Sigma}$ and that the set $\Phi(S_{\Phi_1}^{\Sigma})$ is a singleton, i.e.,
$\Phi(S_{\Phi_1}^{\Sigma})=\{y_0\}$ for some $y_0\in\RR^k$.
Then, every fiber of $\Phi$, other than $\Phi^{-1}(y_0)$, is displaceable from $\Sigma$.
In particular, the fiber $\Phi^{-1}(y_0)$ is a $\Sigma$-stem.
Hence, by Theorem \ref{theorem:X-stem is shv}, $\Phi^{-1}(y_0)$ is $\zeta$-superheavy
for any partial symplectic quasi-state $\zeta$ on $(T^*N,\omega_0)$ such that $\Sigma$ is $\zeta$-superheavy.
\end{theorem}

Therefore, applying Theorem \ref{thm:gen} for $\Sigma=0_N$ 
yields Theorem \ref{main theorem2}.
To prove Theorem \ref{thm:gen}, we require the following lemma.

\begin{lemma}\label{maintheorem}
Let $\Sigma$ be a compact subset of $T^*N$
and $H\colon T^*N\to\RR$ a Hamiltonian satisfying condition $(\star)_{\Sigma}$.
Then, for any $c\in\RR$ with $c<m_H$,
the level set $H^{-1}(c)\subset T^*N$ is displaceable from $\Sigma$.
\end{lemma}

Before proving Lemma \ref{maintheorem}, we show the following well-known fact.

\begin{lemma}\label{rel graph}
Let $X$ be a compact subset of $T^*N$ and $f\colon N\to\RR$ a smooth function on $N$.
Then the set
\[
	\Gamma_f(X)=\{\,(q,p+df_q)\in T^*N\mid (q,p)\in X\,\}
\]
is Hamiltonian isotopic to $X$.
\end{lemma}

\begin{proof}
Let $f\colon N\to\RR$ be a smooth function.
Let $U\subset T^*N$ be an open neighborhood of $\bigcup_{t\in [0,1]}\Gamma_{tf}(X)$.
Choose a smooth function $\rho\colon T^*N\to\RR$ with compact support such that $\rho|_{U}\equiv 1$.
Then, the (time-independent) Hamiltonian $\rho\cdot(f\circ\pi)\colon T^*N\to\RR$
has a compact support and gives a desired Hamiltonian isotopy between $X$ and $\Gamma_f(X)$.
Indeed, for any $(q,p)\in X$ and any $t\in [0,1]$,
\[
	\varphi_{\rho\cdot(f\circ\pi)}^t(q,p)=(q,p+t\cdot df_q)
\]
and hence $\varphi_{\rho\cdot(f\circ\pi)}(X)=\Gamma_f(X)$.
This finishes the proof of Lemma \ref{rel graph}.
\end{proof}

To prove Lemma \ref{maintheorem}, we use a generalized version of Contreras' argument {\cite[Proposition 8.2]{Co06}.

\begin{proof}[Proof of Lemma \ref{maintheorem}]
Let $g$ be a Riemannian metric on $N$.
By condition $(\star)_{\Sigma}$, for each $q\in N$ the restricted funtion $H|_{T^*_qN\cap\Sigma}$ is constant
and let $c_q$ denote that constant.
Then, $m_H=\max_{q\in N}c_q$.

Choose $c\in\RR$ such that $c< m_H$.
By Lemma \ref{rel graph}, it is sufficient to prove that there exists a function $f^\prime\colon N\to\RR$ such that
$\Gamma_{f^\prime}(\Sigma)\cap H^{-1}(c)=\emptyset$.

Take a non-empty open subset $U$ of $N$ so that $\{c_q\}_{q\in U}\subset (c,m_H]$.
Choose a smooth function $f \colon N\to \RR$ whose critical points are contained in $U$.
Since $N\setminus U$ is compact and $df_q\neq0$ for any $q\in N\setminus U$, the number $R_1=\min_{q\in N\setminus U}\| df_q\|_g$ is positive.
We set $\Sigma|_{N\setminus U}=\Sigma\cap T^*N|_{N\setminus U}$ where $T^*N|_{N\setminus U}\subset T^*N$ is the subbundle restricted to $N\setminus U$.
By condition $(\star)_{\Sigma}$, the sets $\Sigma|_{N\setminus U}$ and $H^{-1}\bigl((-\infty,c]\bigr)$ are compact.
Hence there exists a positive number $R_2$ such that
\[
	\Sigma|_{N\setminus U}\cup H^{-1}\bigl((-\infty,c]\bigr)\subset\left.B^*_{g,R_2}N\right|_{N\setminus U}.
\]
We set $R_3=2R_2/R_1$.
Now we claim that
\begin{equation}\label{eq:Gamma N-U}
	\Gamma_{R_3f}(\Sigma|_{N\setminus U})\cap H^{-1}\bigl((-\infty,c]\bigr)=\emptyset.
\end{equation}
By the choice of $R_2$, it is enough to show that
\begin{equation}\label{eq:Gamma ball}
	\Gamma_{R_3f}\left(\left.B^*_{g,R_2}N\right|_{N\setminus U}\right)\cap\left.B^*_{g,R_2}N\right|_{N\setminus U}=\emptyset.
\end{equation}
Arguing by contradiction, assume that there exists a point $(q_0,p_0)$ in the left hand side of \eqref{eq:Gamma ball}.
Recall that
\[
	\Gamma_{R_3f}\left(\left.B^*_{g,R_2}N\right|_{N\setminus U}\right)=\left\{\,(q,p+R_3\cdot df_q)\in T^*N\relmiddle| (q,p)\in \left.B^*_{g,R_2}N\right|_{N\setminus U}\,\right\}.
\]
Since $(q_0,p_0)\in\Gamma_{R_3f}\left(\left.B^*_{g,R_2}N\right|_{N\setminus U}\right)$, we have
$\| R_3\cdot df_{q_0}-p_0\|_g<R_2$.
Since $(q_0,p_0)\in\left.B^*_{g,R_2}N\right|_{N\setminus U}$, we have
$\|p_0\|_g<R_2$.
Thus, by the triangle inequality,
\[
	\| R_3\cdot df_{q_0}\|_g\leq\| R_3\cdot df_{q_0}-p_0\|_g+\| p_0\|_g<R_2+R_2=2R_2.
\]
Therefore, by the choice of $R_1$ and the definition of $R_3$, we have
\[
	R_1\leq\| df_{q_0}\|_g<\frac{2R_2}{R_3}=R_1,
\]
and we obtain a contradiction.
Therefore, \eqref{eq:Gamma N-U} holds.

Let $q\in U$.
By condition $(\star)_{\Sigma}$ and $\{c_q\}_{q\in U}\subset (c,m_H]$,
for any $p\in T^*_qN\cap\Sigma$ we have
\[
	H\bigl(q,p+R_3\cdot df_q\bigr)\geq H(q,p)=c_q>c.
\]
Namely, $\Gamma_{R_3f}(\Sigma|_U)\cap H^{-1}\bigl((-\infty,c]\bigr)=\emptyset$.

Combining with \eqref{eq:Gamma N-U}, we conclude that $\Gamma_{R_3f}(\Sigma)\cap H^{-1}\bigl((-\infty,c]\bigr)=\emptyset$.
In particular, $H^{-1}(c)$ is displaceable from $\Gamma_{R_3f}(\Sigma)$.
By Lemma \ref{rel graph}, $H^{-1}(c)$ is displaceable from $\Sigma$.
This completes the proof of Lemma \ref{maintheorem}.
\end{proof}

\begin{remark}\label{KS remark}
When the authors first found and proved Lemma \ref{maintheorem}, they did not know Contreras' argument.
Seongchan Kim pointed out that Contreras had already used a similar technique.
They would like to thank his pointing out.
\end{remark}

Now we are in a position to prove Theorem \ref{thm:gen}.

\begin{proof}[Proof of Theorem \ref{thm:gen}]
Let $y=(y^1,\ldots,y^k)\in\RR^k$.
If $y\in\RR^k\setminus\Phi(\Sigma)$, then $\Phi^{-1}(y)\cap\Sigma=\emptyset$.
In particular, the fiber $\Phi^{-1}(y)$ is displaceable from $\Sigma$.

Assume that $y\in\Phi(\Sigma)$.
Then, in particular, $y^1\in\Phi_1(\Sigma)$.
Since $\Phi_1$ satisfies condition $(\star)_{\Sigma}$,
for each $q\in N$, the function $\Phi_1|_{T^*_qN\cap\Sigma}$ is constant.
Since $y^1\in\Phi_1(\Sigma)$,
\begin{equation}\label{eq:y_1}
	y^1\leq \max_{\Sigma}\Phi_1=\max_{q\in N}\Phi_1|_{T_q^*N\cap\Sigma}=m_{\Phi_1}.
\end{equation}
If $y^1\neq m_{\Phi_1}$, then \eqref{eq:y_1} and Lemma \ref{maintheorem} imply that
$\Phi_1^{-1}(y^1)$ is displaceable from $\Sigma$ and hence so is $\Phi^{-1}(y)\subset\Phi_1^{-1}(y^1)$.

If $y^1=m_{\Phi_1}$, then
\[
	\Phi^{-1}(y)\cap\Sigma\subset\Phi_1^{-1}(m_{\Phi_1})\cap\Sigma=S_{\Phi_1}^{\Sigma}.
\]
Since we have assumed $y\in\Phi(\Sigma)$ and $\Phi(S_{\Phi_1}^{\Sigma})=\{y_0\}$,
\[
	\{y\}=\Phi\bigl(\Phi^{-1}(y)\cap\Sigma\bigr)\subset\Phi(S_{\Phi_1}^{\Sigma})=\{y_0\}.
\]
Hence $y=y_0$.

Therefore, the above argument implies that every fiber of $\Phi$, other than $\Phi^{-1}(y_0)$, is displaceable from $\Sigma$.
By condition $(\star)_{\Sigma}$, the sublevel set $\Phi_1^{-1}\bigl((-\infty, m_{\Phi_1}]\bigr)$ is compact and hence so is the fiber $\Phi^{-1}(y_0)\subset\Phi_1^{-1}\bigl((-\infty, m_{\Phi_1}]\bigr)$.
Therefore, $\Phi^{-1}(y_0)$ is a $\Sigma$-stem.
This finishes the proof of Theorem \ref{thm:gen}.
\end{proof}


\subsection{Proof of Theorem \ref{secondcorollary}}\label{proofcor}

\begin{proof}
Take $y=(y^1,\ldots,y^k)\in\Phi(T^*N)\subset\RR^k$, where $\Phi=(H_1,\ldots,H_k)\colon T^*N \to \RR^k.$
If $y^i>m_{H_i}$ for some $i\in\{1,\ldots,k\}$,
then $H_i^{-1}(y^i)$ is disjoint from the zero-section $0_N$ and hence so is $\Phi^{-1}(y)\subset H_i^{-1}(y^i)$.
If $y^i<m_{H_i}$ for some $i\in\{1,\ldots,k\}$,
then applying Lemma \ref{maintheorem} for $\Sigma=0_N$,
$H_i^{-1}(y^i)$ is displaceable from $0_N$ and hence so is $\Phi^{-1}(y)\subset H_i^{-1}(y^i)$.

The above argument then implies that every fiber of $\Phi$, other than $\Phi^{-1}(m_{\Phi})$, is displaceable from $0_N$, 
where $m_{\Phi}=(m_{H_1},\ldots,m_{H_k})\in\RR^k$.
By Corollary \ref{pshv open}, $\Phi^{-1}(m_{\Phi})$ is non-displaceable from $0_N$.
Thus,
\[
	\bigcap_{i=1}^k S_{H_i}=\bigcap_{i=1}^k H_i^{-1}(m_{H_i})\cap 0_N=\Phi^{-1}(m_{\Phi})\cap 0_N\neq\emptyset.
\]
This completes the proof of Theorem \ref{secondcorollary}.
\end{proof}


\subsection{Proof of Proposition \ref{maintopo}}\label{proofprop}

Proposition \ref{maintopo} immediately follows from Theorem \ref{prop:shv is hv} (iii), Proposition \ref{zero-sec} and the following assertion.

\begin{proposition}\label{non-surj}
Let $X$ be a compact subset of $T^*N$.
If $\pi(X)\neq N$,
then $X$ is displaceable from the zero-section $0_N$.
\end{proposition}

\begin{proof}
By Lemma \ref{rel graph}, it is enough to show that $\Gamma_f(X)$ is displaceable from $0_N$ for some smooth function $f\colon N\to\RR$.
Let $f\colon N\to\RR$ be a smooth function whose critical points are all contained in $N\setminus\pi(X)$.
Then $df_q\neq 0$ for any $(q,p)\in X$.
Since $X$ is compact,
there exists a positive number $R_0>0$ such that for any $(q,p)\in X$, $R_0\cdot df_q\neq -p$.
It means that
\[
	\Gamma_{R_0f}(X)\cap 0_N=\emptyset.
\]
This completes the proof of Proposition \ref{non-surj}.
\end{proof}


\section{Proof of Proposition \ref{rabi mugen}}\label{Rabi section}

In this section, we prove Proposition \ref{rabi mugen} and provide another corollary (Corollary \ref{cor:CFP Sigma}) of Theorem \ref{thm:gen}.
Under the assumption of Proposition \ref{rabi mugen}, there are many disjoint superheavy subsets in $T^\ast N$.
We use these superheavy subsets to prove the existence of many non-displaceable fibers.
This idea comes from \cite{KO19b}.

\begin{proof}[Proof of Proposition \ref{rabi mugen}]
Arguing by contradiction, assume that the moment map $\Phi$ has finitely many non-displaceable fibers.
Let $\Phi^{-1}(y_1),\ldots,\Phi^{-1}(y_{\ell})$ be all the non-displaceable fibers of $\Phi$,
where $y_1,\ldots,y_{\ell}\in\RR^k$.
By the assumption on $H$, the fibers $\Phi^{-1}(y_i)$, $i=1,\ldots,\ell$, are compact.
Then there exists a positive number $r$ such that
\begin{equation}\label{eq:big ball}
	\bigcup_{i=1}^{\ell}\Phi^{-1}(y_i)\subset B^*_{g,r}N.
\end{equation}
By assumption, there exist a positive number $R$ with $R>r$
and a partial symplectic quasi-state $\zeta_R\colon C_c(T^\ast N)\to\mathbb{R}$ such that $S^\ast_{g,R}N$ is $\zeta_R$-superheavy.
Then, by \eqref{eq:big ball},
\begin{equation}\label{eq:majiwaru}
	\left(\bigcup_{i=1}^{\ell}\Phi^{-1}(y_i)\right)\cap S^\ast_{g,R}N=\emptyset.
\end{equation}

Since $S^\ast_{g,R}N$ is $\zeta_R$-superheavy, by Corollary \ref{pshv open},
there exists $y_0\in\Phi(T^*N)$ such that
the fiber $\Phi^{-1}(y_0)$ is non-displaceable from itself and from $S^\ast_{g,R}N$.
Therefore, $y_0\in\{y_1,\ldots,y_{\ell}\}$ and $\Phi^{-1}(y_0)\cap S^\ast_{g,R}N\neq\emptyset$.
It contradicts \eqref{eq:majiwaru} and we complete the proof of Proposition \ref{rabi mugen}.
\end{proof}

Moreover, we have the following corollary of Theorem \ref{thm:gen}.

\begin{corollary}\label{cor:CFP Sigma}
Let $N$ be a closed manifold, $\Sigma$ a compact subset of $T^*N$,
and $H\colon T^*N\to\RR$ a Hamiltonian satisfying condition $(\star)_{\Sigma}$.
Assume that there exists a partial symplectic quasi-state $\zeta\colon C_c(T^*N)\to\RR$ on $(T^*N,\omega_0)$ such that $\Sigma$ is $\zeta$-superheavy.
Then, the level set $H^{-1}(m_H)\subset T^*N$ is non-displaceable from itself and from $\Sigma$.
\end{corollary}

\begin{proof}
By Theorem \ref{thm:gen}, the level set $H^{-1}(m_H)$ is a $\Sigma$-stem.
By Corollary \ref{pshv open}, $H^{-1}(m_H)$ is non-displaceable from itself and from $\Sigma$.
\end{proof}

We provide an example of Corollary \ref{cor:CFP Sigma}.

\begin{example}
Let $(N,g)$ be a closed Riemannian manifold and $r$ a non-negative number.
Let $H\colon T^*N\to\RR$ be a Hamiltonian of the form
\[
	H(q,p)=\rho(\| p\|^2)+U(q),
\]
where $U\colon N\to\RR$ is a smooth potential and $\rho\colon [0,\infty)\to\RR$ is a smooth function which attains its minimum value at $r^2$ and satisfies $\lim_{x\to+\infty}\rho(x)=+\infty$.
Then $H$ satisfies condition $(\star)_{\Sigma}$ where $\Sigma=S^*_{g,r}N$.
Assume that there exists a partial symplectic quasi-state $\zeta\colon C_c(T^*N)\to\RR$ on $(T^*N,\omega_0)$ such that $S^*_{g,r}N$ is $\zeta$-superheavy.
Then, by Corollary \ref{cor:CFP Sigma},
the level set $H^{-1}(m_H)\subset T^*N$ is non-displaceable from itself and from $S^*_{g,r}N$.
\end{example}

 
\section*{Acknowledgments}

The authors cordially thank Felix Schlenk for reading a preliminary version very carefully and for giving very helpful comments.
They also thank Seongchan Kim for giving them the trigger to study the present topic.
When they talked with him about a different mathematical topic, he explained how interesting the spherical pendulum is to them.
This started the first author to consider a superheavy fiber of the spherical pendulum.
He also suggested some other integrable systems (Examples \ref{pendulum} and \ref{neu}) and gave remarks (Remarks \ref{Pol torus} and \ref{KS remark}).
They also sincerely thank Mitsuaki Kimura, Takahiro Matsushita, and Yuhei Suzuki for fruitful discussions and warmhearted advices.

This work has been supported by JSPS KAKENHI Grant Numbers JP18J00765, JP18J00335.


\bibliographystyle{amsart}

\end{document}